\documentclass[11pt]{article}

\usepackage{cancel,amsthm,amsmath,times,amsopn,amsfonts,amssymb,epsfig,anysize}
\usepackage{tabmac,multicol}
\usepackage{graphicx}
\usepackage[all]{xy}
\usepackage{pifont,txfonts,graphicx,geometry}
\usepackage[square,numbers]{natbib}
\usepackage{pb-diagram}

\usepackage[vcentermath]{youngtab}
\usepackage{amsmath}
\usepackage{amsthm} 
\usepackage{verbatim}

\usepackage{latexsym,graphics,color}
\usepackage{tabmac_avi}
\usepackage[usenames,dvipsnames]{xcolor}
\usepackage{algorithmic}

\newtheorem{theorem}{Theorem}
\newtheorem{lemma}[theorem]{Lemma}
\newtheorem{prop}[theorem]{Proposition}

\newtheorem{example}[theorem]{Example}
\newtheorem{Ex}[theorem]{Example}

\theoremstyle{definition}
\newtheorem{definition}[theorem]{Definition}

\theoremstyle{remark}
\newtheorem{remark}[theorem]{Remark}

\newcommand{\SSYT}{\mathcal{SSYT}}
\newcommand{\SVT}{{\mathcal{SVT}\!\!}}
\newcommand{\RPP}{{\rm RPP}}
\newcommand{\EF}{{\rm EF}}
\newcommand{\RSK}{{\rm RSK}}

\renewcommand{\th}{^{\text{th}}}

\newcommand{\bx} {\mathbf{x}}

\def\shape{ {\rm {shape}}}

\def\fw { {\rm {sw}}}
\def\T { {\mathcal T}}

\def \wt {{\rm wt}}
\def \wor {{\rm w}}
\def \sh {{\rm shape}}

\def \fwt {\iota{\rm wt}}

\begin{document}

\begin{center}
{\Large Structure constants for $K$-theory of Grassmannians, revisited}

\bigskip
\bigskip

{\large $\text{Huilan Li}^\dagger$, $\text{Jennifer Morse}^\ddagger$\footnote{This author is supported
			by NSF:1301695}, and $\text{Patrick Shields}^\ddagger$\footnote{
			This author is supported by NSF:1301695 }
			}

\bigskip

\begin{multicols}{2}

{\it $^\dagger$Shandong Normal University

School of Mathematical Sciences

Jinan, Shandong 250014, China}

{\it $^\ddagger$Drexel University

Department of Mathematics

Philadelphia, PA 19104 }

\end{multicols}

\end{center}
\begin{abstract}
The problem of computing products of Schubert classes in the cohomology 
ring can be formulated as the problem of expanding skew Schur polynomials
into the basis of ordinary Schur polynomials.   
In contrast, the problem 
of computing the structure constants of the Grothendieck ring of a 
Grassmannian variety with respect to its basis of Schubert structure sheaves
is not equivalent to expanding skew stable Grothendieck polynomials
into the basis of ordinary stable Grothendiecks.
Instead, we show that the appropriate $K$-theoretic analogy 
is through the expansion of skew reverse plane partitions into the basis 
of polynomials which are Hopf-dual to stable Grothendieck polynomials.
We combinatorially prove this expansion is determined by Yamanouchi 
set-valued tableaux.  A by-product of our results is a dual approach
proof for Buch's $K$-theoretic Littlewood-Richardson rule for the product of 
stable Grothendieck polynomials.
\end{abstract}


\section{Introduction}
The theory of symmetric functions supports Schubert calculus
of the Grassmannian $X={\rm Gr}(k,n)$ by way of the Schur function basis.
In particular, the Schubert decomposition of $X$ is determined by
varieties $X_\lambda$, one for each partition $\lambda$ in a 
$k\times (n-k)$ rectangle.  The Schubert varieties induce 
canonical basis elements for the cohomology ring
of $X$.  In turn, $H^*(X)$ is isomorphic to a certain quotient of the ring 
of symmetric functions $\Lambda$ in $k$ variables.  The Schur functions
$s_\lambda(x_1,\ldots,x_k)$ are representatives for the Schubert 
classes with a critical feature that structure constants in
$$
[X_\lambda]\cdot[X_\mu] = \sum_{\nu} 
C_{\lambda\mu}^\nu \,[X_\nu]
$$
appear as Schur coefficients in a product of Schur functions:
$$
s_\lambda\,s_\mu = \sum_\nu C_{\lambda\mu}^\nu\,s_{\nu}\,.
$$

The realization of Schur functions as weight generating functions of 
semi-standard Young tableaux then offers combinatorial tools for the study.
The development of a rich theory of tableaux ultimately enabled 
the computation of the structure constants $C_{\lambda\mu}^\nu$.
The Littlewood-Richardson rule for $C_{\lambda\mu}^\nu$ dictates
a count of tableaux with certain restrictions~\cite{LR}, and its proof
was settled in~\cite{S77} using the RSK insertion algorithm
on tableaux~\cite{R38,S61,K70}.

As a basis of symmetric functions, the Schur functions are self-dual
with respect to the Hall inner product on $\Lambda$.   This duality 
is central in symmetric function theory and in particular, gives
an alternative way to access the Schubert structure constants.  Namely, 
Littlewood-Richardson numbers arise in the
Schur expansion of skew Schur functions:
\begin{equation}
\label{skewschur}
s_{\nu/\lambda} = \sum_\mu C_{\lambda\mu}^\nu \, s_\mu\,.
\end{equation}
This viewpoint gives rise to simpler proofs of the 
Littlewood-Richardson rule, e.g. \cite{RS98,Ste02}.

Developments in Schubert calculus have established the
importance of combinatorics and symmetric function theory to the
more intricate setting of the Grothendieck ring $K^\circ X$ of 
algebraic vector bundles on $X$.
Lascoux and Sch\"utzenberger~\cite{LS} introduced Grothendieck polynomials 
as representatives for the structure sheaves of the Schubert varieties
in a flag variety.
Fomin and Kirillov~\cite{FK94} studied the symmetric power series resulting from a 
limit of Grothendieck polynomials.   

For the Grassmannian variety $X$,
it was shown in~\cite{Buch} that the stable limits $G_\lambda$ 
are generating series of set-valued tableaux and can be applied to
$K^\circ X$ in a way that mirrors the Schur role in cohomology.  
The classes of the Schubert structure sheaves $\mathcal O_{X_\lambda}$ 
form a basis for the Grothendieck ring of $X$, and the
structure constants in 
$$
[\mathcal O_{X_\lambda}]\cdot [\mathcal O_{X_\mu}]
=\sum_\nu c_{\lambda\mu}^\nu\, [\mathcal O_{X_\nu}]
$$
appear in the product
\begin{equation}\label{Klr}
G_\lambda \, G_\mu = \sum_\nu 
c_{\lambda\mu}^\nu \,G_\nu\,.
\end{equation}
By generalizing RSK insertion, Buch proved that the expansion
coefficients are determined by the number of 
set-valued tableaux with a Yamanouchi column word.

In contrast to~\eqref{skewschur}, the Schubert structure constant
$c_{\lambda\mu}^\nu$ for $K^\circ X$ is not the coefficient of $G_\mu$ 
in $G_{\nu/\lambda}$.  Instead, the bialgebra $\Gamma={\rm span}\{G_\lambda\}$ 
gives rise to a distinct family of symmetric functions $\{g_\lambda\}$ 
defined by taking the Hopf-dual basis to $\{G_\lambda\}$.  Nevertheless, 
inspired by the far-reaching applications of Schur function duality, 
we have discovered that duality can be used to access the $K$-theoretic constants.

A combinatorial investigation of Lam and Pylyavskyy~\cite{LP07} revealed that 
the basis $\{g_\lambda\}$ can be realized as the weight generating functions
of reverse plane partitions.   We study the skew generating functions,
$$
g_{\nu/\lambda} = \sum_{R\in \RPP(\nu/\lambda)} {\bf x}^{\wt(R)}\,,
$$
proving that their $g$-expansion coefficients,
\begin{equation}\label{skewg}
g_{\nu/\lambda} = \sum_\mu a_{\lambda\mu}^\nu \,g_\mu\,,
\end{equation}
match the $K$-theoretic structure constants $c_{\lambda\mu}^\nu$.

We bijectively convert the reverse plane partitions into
tabloids, classical combinatorial objects dating back to Young's study 
of irreducible $S_n$-representations.  
This identification requires an unconventional notion of tabloid weight 
called {\it inflated weight} which we relate to the Yamanouchi property on 
set-valued tableaux.  From the tabloid characterization for $g_{\nu/\lambda}$,
a crystal structure on tabloids~\cite{genomic} leads us to its Schur expansion 
in terms of a distinguished subclass of semi-standard tableaux.  
In turn, the Schur expansion can be converted into a $g$-expansion using elegant 
fillings with the introduction of a sign ~\cite{Lenart}.  
A bijection $\phi_T$ on pairs of tableaux and elegant fillings
leads us to an expression over certain set-valued tableaux which we
then reduce by way of a sign-reversing involution $\tau$.
This establishes that the $g$-expansion coefficients for $g_{\nu/\lambda}$ count
the same Yamanouchi set-valued tableaux prescribed by Buch in 
his $K$-theoretic Littlewood-Richardson rule for $c_{\lambda\mu}^\nu$.
As a by-product, \eqref{skewg} is equivalent to \eqref{Klr}.

\section{Preliminaries}
\label{sec:definitions}

Let $X={\rm Gr}(k,n)$ denote the Grassmannian of $k$-dimensional
subspaces of $\mathbb C^n$ and consider the subspaces $\mathbb C^d$
of vectors which can have non-zero entries only in the first $d$ components.
$X$ can be decomposed into a family of distinguished ``Schubert cells",
indexed by non-decreasing integer partitions $\lambda=(\lambda_1,\ldots,\lambda_k)$ 
with $k$ parts, none of which are larger than $n-k$.
Their closures are the Schubert varieties
$$
X_\lambda = \{V\in {\rm Gr}(k,n) : {\rm dim}(V\cap \mathbb C^{n-k+i-\lambda_i})\geq i
,\ 
\forall 1\leq i\leq k\}\,.
$$

The classes of Schubert cells form a basis for the cohomology of $X$.
On the other hand, $H^*(X)$ is isomorphic to a certain quotient of the ring of 
symmetric polynomials $\Lambda_k=\mathbb Z[e_1,\ldots,e_k]\\=\mathbb Z[h_1,\ldots,h_k]$, 
where
$$
e_r = \sum_{1\leq i_1<\cdots<i_r} x_{i_1}\cdots x_{i_r}\qquad\text{and}
\qquad
h_r = \sum_{1\leq i_1\leq \cdots \leq i_r} x_{i_1}\cdots x_{i_r}\,.
$$
Bases for $\Lambda$ are indexed by generic partitions
$\lambda=(\lambda_1,\lambda_2,\ldots,\lambda_\ell)$.
A quick example is the basis of homogenous symmetric functions $\{h_\lambda\}$, defined by
$h_\lambda = h_{\lambda_1}\cdots h_{\lambda_\ell}$.

The combinatorial potential of $\Lambda$ is met with the unique association 
of each partition $\lambda$ with its {Ferrers 
shape}, a left-and bottom-justified array of $1 \times 1$ square cells in the first quadrant of the 
coordinate plane, with $\lambda_i$ cells in the $i^{th}$ row from the bottom.  
Given a partition $\lambda$, its conjugate $\lambda'$ is the partition
obtained by reflecting the shape of $\lambda$ about the line $y=x$.  
For partitions $\mu,\lambda$, the property that every cell of $\mu$ 
is also a cell of $\lambda$ is denoted by $\mu\subset\lambda$.  
For $\mu\subset\lambda$, the skew shape $\lambda/\mu$
is defined by the cells in $\lambda$ but not in $\mu$.
The skew shape obtained by placing the diagram of a partition $\lambda$
southeast and caty-corner to a partition $\mu$ is denoted by
$\mu*\lambda$.
For example,
$$
\mu=(2,2,1), \lambda=(3,1)\,\implies\,
\mu*\lambda\,=\,(5,3,2,2,1)/(2,2)\,=\,
{\huge\tableau[pcY]{\cr&\cr&\cr\bl&\bl&\cr\bl&\bl&&&}}
$$
More generally, a {composition} is a sequence of non-negative
integers $\alpha = (\alpha_1, \alpha_2, \ldots, \alpha_k)$,
and $\alpha-\beta$ is defined by usual vector subtraction
for any compositions $\alpha,\beta$. For any composition $\alpha$, let $|\alpha|=\alpha_1+\alpha_2+\cdots+\alpha_k$.

A {\it(semi-standard) tableau} of {\it shape} $\lambda$ is a positive
integer filling of the cells of $\lambda$ such that entries weakly increase from left to right in rows and
are increasing from the bottom to top of each column.
The {\it weight} of a tableau $T$ denoted  $\wt(T)$ is the composition
$\alpha=(\alpha_i)_{i \geq 1}$ where $\alpha_{i}$ is the number of cells
containing $i$ (it is customary to omit trailing $0$'s).  
For any partition $\lambda$, there is a unique tableau 
of both shape and weight $\lambda$.  We denote this tableau
by $T_\lambda$.  We use
$\SSYT(\lambda)$ to denote the set of all semi-standard tableaux of
shape $\lambda$, and $\SSYT(\lambda,\mu)$ to denote the set of all
semi-standard tableaux of shape $\lambda$ and weight $\mu$.

Schur functions may be defined as the weight generating functions of semi-standard tableaux; for any partition $\lambda$,
  \[ s_\lambda = \sum_{T \in \SSYT(\lambda)}^{} \bx^{\wt(T)} \;,\]
where $\bx^{\alpha} = x_1^{\alpha_1} x_2^{\alpha_2} \cdots$.
They are symmetric functions and form a basis for $\Lambda$.
As such, it is convenient to collect terms into the basis of monomial
symmetric functions,
formed by $m_\lambda = \sum_{\alpha}^{} \bx^{\alpha}$
over all distinct rearrangements $\alpha$ of partition $\lambda$.
The monomial expansion of a Schur function is 
\begin{align} \label{eq:stom}
  s_{\lambda} = \sum_{\mu}^{} K_{\lambda\mu}\, m_{\mu}\,,
\end{align}
where the Kostka coefficients, $K_{\lambda,\mu}$,
enumerate tableaux of shape $\lambda$ and weight $\mu$. 

The Schur basis is orthonormal with respect to the Hall inner product
$\langle,\rangle$ on $\Lambda$, defined by
\begin{equation}
\label{sdual}
\langle m_\lambda,h_\mu\rangle = 
  \begin{cases}
    1 \quad \text{ if } \lambda=\mu \\ 0 \quad \text{otherwise.}
  \end{cases}
\end{equation}
An immediate consequence of (\ref{eq:stom}) and duality is 
\begin{align} \label{eq:htos}
  h_\mu = \sum_{\lambda}^{} K_{\lambda\mu}\, s_\lambda \,.
\end{align}

Our methods rely on fundamental operations on words and tableaux
including {jeu de taquin} and $\RSK$-insertion~\cite{S77,R38,S61,K70}.
We briefly recall several important results here;
full details can be found in a variety of texts such as~\cite{LSmonoid,EC2,Ful}.  
A word $w$ is {\it Yamanouchi} when each factorization $w=uv$
satisfies ${\wt}_i(v)\geq {\wt}_{i+1}(v)$ for all $i$, where ${\wt}_i(v)$ 
is the number of times letter $i$ appears in $v$.  For a given partition $\lambda$, 
a word $w$ is {\it $\lambda$-Yamanouchi} if ${\wt}_i(v)+\lambda_i\geq {\wt}_{i+1}(v)+
\lambda_{i+1}$.

The \emph{row word} $\wor(T)=w_1 w_2 \cdots w_n$ of tableau $T$ 
is defined by listing elements of $T$ starting from the top-left corner, 
reading across each row, and then continuing down the rows.  
We say that tableau $T$ is row Yamanouchi when its row word is Yamanouchi.
The RSK insertion algorithm uniquely identifies a word $w$ with 
a pair of same shaped tableaux $(P(w),Q(w))$.  
For any tableau $T$,
\[ P(\wor(T)) = T\,, \]
and when two words $w$ and $u$ have the same insertion tableau $P(w)=P(u)$,
they are \emph{Knuth equivalent}, denoted by $w\sim u$. 


\section{$K$-theoretic Littlewood-Richardson coefficients}

Because the Schubert cells form a cell decomposition of the
Grassmannian $X$, the classes of the structure 
sheaves $\mathcal O_{X_\lambda}$ form a basis for the
Grothendieck ring of $X$.  A combinatorial description for
the structure constants of $K^\circ X$ with respect to its basis 
of Schubert structure sheaves, appearing in 
\begin{equation}
\label{eq:sslr}
[\mathcal O_{X_\lambda}]\cdot[\mathcal O_{X_\mu}] = \sum_\nu
c_{\lambda\mu}^\nu\,[\mathcal O_{X_\nu}]\,,
\end{equation}
was first given in~\cite{Buch}.

\subsection{Combinatorial $K$-theoretic background}

Buch's work initiated a framework for $K$-theoretic Schubert calculus using 
a set-valued generalization of tableaux.  
A {\it set-valued tableau} of shape $\nu/\lambda$ is defined to be
a filling of each cell in the diagram of $\nu/\lambda$ with a non-empty set of positive
integers such that each subfilling created 
from the choice of a single element in each cell is a semi-standard
tableau.
The {\it weight} of a set-valued tableau $S$ is the composition
$\wt(S)=(\alpha_i)_{i \ge 1}$ where $\alpha_i$ is the total number
of times $i$ appears in $S$ and the {\it excess} of $S$ is defined 
by
\[ \varepsilon(S) = |\wt(S)| - |\shape(S)| \;. \]
A {\it multicell} refers to a cell in $S$ that contains more than
one letter.  When $S$ has no multicells, it is viewed as a 
semi-standard tableau. In this case, $|\wt(S)|=|\shape(S)|$ and
$\varepsilon(S) = 0$. 

The collection of all set-valued tableaux of shape $\nu/\lambda$ is 
denoted by $\SVT(\nu/\lambda)$ and the subset of these 
with weight $\alpha$ is $\SVT(\nu/\lambda,\alpha)$.  
For any skew shape $\nu/\lambda$, the {\it stable Grothendieck polynomial}
is defined by
\begin{equation}
\label{the:Gbuch}
    G_{\nu/\lambda} = \sum_{\mu}^{} k_{\nu/\lambda,\mu}\, m_\mu
\,,
\end{equation}
where
$$
k_{\nu/\lambda,\mu} = 
 (-1)^{|\lambda|-|\nu|-|\mu|} 
\sum_{S\in \SVT(\nu/\lambda,\mu)} 1\,.
$$
Through the study of the vector space $\Gamma={\rm span}\{G_\lambda\}$ over
all partitions $\lambda$, Buch proved that the set-valued  generating functions 
give a set of representatives for the classes of $\mathcal O_{X_\lambda}$.
\begin{theorem}[\cite{Buch}]
\label{the:buch}
The structure constants of~\eqref{eq:sslr} occur as coefficients 
in the expansion of the symmetric function product
\begin{equation}
G_\lambda \, G_\mu = \sum_{\nu\atop |\nu|\geq|\lambda|+|\mu|} c_{\lambda\mu}^\nu \,G_\nu\,.
\end{equation}
\end{theorem}

This inspired Buch to define a column insertion algorithm on set-valued tableaux and 
to give combinatorial meaning to $c_{\lambda\mu}^\nu$; they count the subset of 
set-valued tableaux with a certain Yamanouchi condition~\cite{Buch}.  
Precisely, the {\it column word} 
$\wor(S)$ of a set-valued tableau $S$ is defined by reading cells from top 
to bottom in columns and taking letters within a cell from smallest to largest,
starting with the leftmost column moving right.  When $\wor(S)$ is $\lambda$-Yamanouchi, 
we say $S$ is a \emph{column $\lambda$-Yamanouchi set-valued tableau}.  

\begin{example}
\[S=\tableau[scY]{5\cr 4& 57\cr 12& 2& 23} 
\implies w(S)= 541257223\]
\end{example}

\begin{theorem}[\cite{Buch}]
\label{thm:20}
For partitions $\lambda,\mu$,
$c_{\lambda\mu}^\nu$  is equal to $(-1)^{|\nu|-|\lambda|-|\mu|}$ 
times the number of column $\lambda$-Yamanouchi set-valued tableaux of shape $\mu$ 
and weight $\nu-\lambda$.
\end{theorem}
The result has since been reproven in a variety of other ways.  Various
combinatorial approaches include~\cite{MC06,IS14,PY15}.

\subsection{Duality}

As discussed, the Littlewood-Richardson coefficients in a product of Schur 
functions are the same as the Schur expansion coefficients of a skew 
Schur function~\eqref{skewschur}.  In contrast, the constant $c_{\lambda\mu}^\nu$ 
in a product of stable Grothendieck polynomials is not the coefficient 
of $G_\mu$ in $G_{\nu/\lambda}$.  However, we have discovered that there is 
a natural remedy.

Lam and Pylyavskyy~\cite{LP07} proved that $\Gamma$ is a Hopf-algebra, and
they studied the family $\{g_\lambda\}$ which is Hopf-dual to $\{G_\lambda\}$.
They proved that this basis, defined by inverting the triangular system
\begin{align} \label{eq:htog2}
  h_\mu &= \sum_{\lambda}^{}  k_{\lambda\mu}
\, g_\lambda\,,
\end{align}
has a direct combinatorial characterization.

Define a {\it reverse plane partition} $R$ of shape $\nu/\lambda$ to 
be a filling of cells in $\nu/\lambda$ with positive integers which are weakly 
increasing in rows and columns.  The {\it weight} of $R$  is
the composition $\wt(R)=(\alpha_i)_{i \ge 1}$ 
where $\alpha_i$ is the total number of columns of $R$ in which $i$ appears.
The collection of all reverse plane partitions of shape $\nu/\lambda$ is denoted 
by $\RPP(\nu/\lambda)$ and the subset of these with weight $\alpha$ is
$\RPP(\nu/\lambda,\alpha)$.  Let $r_{\nu/\lambda,\alpha}=|\RPP(\nu/\lambda,\alpha)|$. 

\begin{theorem}[\cite{LP07}] \label{Thm:LP}
For any partition $\lambda$,
  \begin{align*}
    g_\lambda &= \sum_{R \in \RPP(\lambda)}^{} \bx^{\wt(R)}=
\sum_{\mu} r_{\lambda\mu}\,m_\mu \;.
  \end{align*}
\end{theorem}

We appeal to duality to realize the $K$-theoretic Littlewood-Richardson
numbers instead as coefficients in the $g$-expansion of 
the skew $g$-functions, defined by
\begin{equation}
\label{skewginm}
g_{\nu/\lambda}=\sum_\mu r_{\nu/\lambda,\mu}\, m_\mu\,.
\end{equation}

\begin{prop}
\label{prop:21}
For partitions $\lambda\subseteq\nu$, 
$$
{g}_{\nu/\lambda} = \sum_{\mu} 
c_{\lambda\mu}^{\nu}\, {g}_\mu \, .
$$
\end{prop}
\begin{proof}
Equivalently we show that
$\langle G_\lambda G_\mu,g_\nu\rangle= \langle g_{\nu/\lambda},G_\mu \rangle$. 
Consider \eqref{skewginm} when $\lambda=\emptyset$ and take the 
inverse to get
$$
m_\mu = \sum_\nu \bar r_{\nu\mu}\,g_\nu
\iff G_\nu = \sum_\mu \bar r_{\nu\mu}\,h_\mu
$$
by duality.
From here,
$$
\langle
g_{\nu/\lambda}, G_\mu
\rangle
=
\langle
\sum_{\beta} 
r_{\nu/\lambda,\beta}\, m_\beta, \,
\sum_{\alpha} \bar r_{\mu\alpha} \,h_\alpha
\rangle
=
\sum_{\alpha} \bar r_{\mu\alpha} \,
r_{\nu/\lambda,\alpha}
\,.
$$
On the other hand, the Pieri rule for $\{G_\lambda\}$ \cite[Theorem 3.2]{Lenart} is
\begin{equation}
 G_\lambda\, h_\alpha = \sum_{\nu} r_{\nu/\lambda,\alpha}\,G_\nu\,,
\end{equation}
implying that
\begin{eqnarray*}
\langle
G_\lambda\,G_\mu,\, g_{\nu}
\rangle
& = &
\langle
\sum_{\alpha} \bar r_{\mu\alpha} \, G_\lambda\,h_\alpha,
\,
g_{\nu}
\rangle
=
\langle
\sum_{\alpha} \bar r_{\mu\alpha}
\sum_{\beta} r_{\beta/\lambda,\alpha} G_\beta,
g_{\nu}
\rangle
\\
& = &
\sum_{\alpha} \bar r_{\mu\alpha}
r_{\nu/\lambda,\alpha}
\end{eqnarray*}
\end{proof}

A by-product of Theorem~\ref{thm:20} and Proposition~\ref{prop:21} is that the 
coefficient of $g_\mu$ in $g_{\nu/\lambda}$ is determined by counting column 
$\lambda$-Yamanouchi set-valued tableaux.  This begs for a direct 
proof, starting from the weight generating functions
\begin{equation}
\label{newskewg}
g_{\nu/\lambda}=\sum_\alpha r_{\nu/\lambda,\alpha}\, {\bf x}^\alpha\,,
\end{equation}
and combinatorirally collecting terms to identify the $g$-expansion coefficients
without relying on Theorem~\ref{thm:20} and Proposition~\ref{prop:21}.
Our method is to use a crystal structure which converts \eqref{newskewg} into
a Schur expansion and to then use the transition between Schur functions and 
the $g$-basis given by Lenart~\cite{Lenart}.
That is, {\it strict elegant fillings}
are skew tableaux with strictly increasing entries up columns and 
across rows, and with the additional property that entries
in row $i$ are not larger than $i-1$.
The set of all strict elegant fillings of shape $\lambda/ \mu$ is
denoted by $\EF(\lambda/\mu)$ and we set $F^\lambda_\mu=|\EF(\lambda/\mu)|$.
\begin{theorem}[\cite{Lenart}] \label{Thm:Lenart}
For any partition $\lambda$,
\begin{align}
  \label{lenthmG}
  G_\mu &= \sum_{\lambda:\mu\subset\lambda}^{} (-1)^{ |\lambda|-|\mu| }\,
    F_{\mu}^{\lambda}\, s_\lambda\,.
\end{align}
An immediate corollary comes out of duality relations:
\begin{align}
\label{sing}
  s_\lambda &= \sum_{\mu:\mu\subset\lambda}^{} (-1)^{|\lambda|-|\mu| }\,
    F_{\mu}^\lambda\, g_\mu\,.
\end{align}
\end{theorem}


\section{Tabloids}

The process of combinatorially tracing this through 
seems daunting upon first glance because two of the transitions
involve signs.  However, converting from the setting of reverse plane 
partitions to one involving a combinatorial structure called tabloid
enables  us to carry it through seamlessly.

\subsection{Inflated weight characterization for Yamanouchi tabloid}

{\it Tabloids} are fillings with positive integers that are weakly increasing in rows.
Although these are classical combinatorial objects going back to Young's study 
of the irreducible representations of $S_n$,
our purposes require an association of tabloids
(and set-valued tableaux) with a less familiar weight called the inflated weight.
We start with the definition and several results establishing that
the notion of this particular weight is closely tied to 
the Yamanouchi property.

Let $\mathcal T(\nu/\lambda)$ denote the set of tabloids with
shape $\nu/\lambda$.  The {\it weight} of a tabloid $T$ is the 
vector $\wt(T)=\alpha$ where $\alpha_i$ records the number of
times $i$ appears in $T$.  The subset $\mathcal T(\nu/\lambda,\alpha)
\subset \mathcal T(\nu/\lambda)$ contains only tabloids of weight $\alpha$. We require another statistic which can be associated to a tabloid to further our investigation. 

\begin{definition}
\label{def:iwt}
The {\it inflated weight} of a tabloid $T$, $\fwt(T)$, is defined iteratively 
from its rows $T_1, T_2,\ldots, T_\ell$, read from bottom to top. 
With $r=2$ and $\hat T=T_1$,
modify $T_r*\hat T$ by moving the last letter $e$ in $T_r$ to the
rightmost empty cell of row $r$ that has no entry $e'\geq e$ 
below it in any row $r'\leq r-1$. If no such cell exists, $e$ remains in place.
Now ignoring $e$, repeat with the last letter in row $r$.
Once all letters in row $r$ have been addressed, iterate by 
setting the resulting filling to $\hat T$ and $r:\!=r+1$.
When $r=\ell$, the process terminates with a column-strict filling $\hat T$ 
called the {\it inflated weight tableau} of $T$. We refer to the entries at the top of each column in $\hat T$ as {\it uncovered}.
The inflated weight is the conjugate of the partition rearrangement of
the uncovered entries in $\hat T$.
\end{definition}

The set of tabloids in $\mathcal T(\nu/\lambda)$  with inflated weight $\alpha$
is denoted by $\mathcal T(\nu/\lambda,\fwt=\alpha)$.

\begin{example}
\[T={\scriptsize\tableau[scY]{1 \cr 2&3&6\cr 1&2&4&5&7} 
\rightarrow \tableau[scY]{1\cr \bl& 2&3&6 \cr \bl&\bl&\bl&\bl&1&2&4&5&7}
\rightarrow \hat T=\tableau[scY]{1\cr \bl& 2&3&\bl&6 \cr \bl&1&2&4&5&7}
}\]
implying that $\fwt(T)=(7,6,4,3,2,1)'=(6,5,4,3,2,2,1)$
\end{example}

This example demonstrates that the weight of a tabloid does not 
necessarily equal its inflated weight.  It is possible for
weight and inflated weight to match; in fact, this characterizes the
Yamanouchi property on tabloid.
The {\it row word} of a tabloid $T$ is defined by reading letters of $T$
from left to right in rows, starting with the top row and proceeding
south.  A tabloid $T$ is {\it row Yamanouchi} when its row word is Yamanouchi.

\begin{example}
For tabloid $T={\scriptsize\tableau[scY]{1\cr1&2}}$\,, consider
$T*T_{(3,1)}$.  Its row word 1122111 is Yamanouchi and
$\wt(T*T_{(3,1)})=(5,2)$.
The inflated weight computation also gives $(5,2)$.
\[T*T_{(31)}={\scriptsize\tableau[scY]{1\cr1&2\cr\bl&\bl&2\cr\bl&\bl&1&1&1} 
\rightarrow \tableau[scY]{1\cr \bl&1&2 \cr \bl&\bl&\bl&2\cr\bl&\bl&\bl&\bl&1&1&1}
\rightarrow \tableau[scY]{1\cr\bl&1&\bl&2\cr \bl&\bl&\bl&\bl&2\cr\bl&\bl&1&1&1}
}
\quad \implies \quad
\fwt(T*T_{(31)})=(2,2,1,1,1)'
\,.
\]
\end{example}

\begin{prop}
\label{lem:tabdweightyam}
For any tabloid $T$ and partition $\lambda$,
$$\text{
$T*T_\lambda$ is row Yamanouchi $\;\iff\;$ $\fwt(T*T_\lambda)=\wt(T*T_\lambda)$.
}
$$
\end{prop}
\begin{proof}
Consider a partition $\lambda=(\lambda_1,\ldots,\lambda_\ell)$.
Take $T*T_\lambda$ to be row Yamanouchi and note that the bottom row 
in its inflated weight tableau $D$ is a row with $\lambda_1$ ones. 
Thus, to prove that $\fwt(T*T_\lambda)=\wt(T*T_\lambda)$, it suffices 
to show that $D$ has {\it interval valued columns}.  That is, 
we claim that the column of an uncovered entry $x$ in $D$ 
contains precisely $1,2,\ldots,x$.

This condition holds on $T_\lambda$ since the rows in its inflated
weight tableau are {\it right} justified and row $i$ is made up of
$\lambda_i$ $i$'s.  Next is the construction of row $\lambda_\ell+1$ in 
the inflated weight tableau of $T*T_\lambda$, beginning by sliding entry 
$b$ into the rightmost cell $(\lambda_\ell+1,j)$ such that all entries 
in column $j$ are smaller than $b$.  
By the Yamanouchi property on $T*T_\lambda$, there are
strictly more $b-1$'s than $b$'s in lower rows
and at least one is uncovered since columns are interval valued.
Therefore, the rightmost column with an uncovered $b-1$ is 
the column $j$ into which $b$ is moved and it remains interval 
valued.  
The same argument applies for the
next letter $a$ in row $\lambda_\ell+1$ since $a\leq b$.  By
iteration, columns in the inflated weight tableau are interval valued.

On the other hand, consider $T*T_\lambda$ with the property that 
$\fwt(T)=\wt(T)$.  Since columns of the inflated weight tableau $D$ 
are strictly increasing, every column with an entry $x$ must contribute 
to the number of uncovered entries in $D$ which are at least $x$.
Therefore, columns of $D$ must be interval valued because
$\fwt(T)=\wt(T)$ implies that the number of uncovered entries 
in $D$ which are at least $x$ equals the number of times $x$ occurs 
in $D$. 
Every $x>1$ in $D$ 
must then lie above an $x-1$ in the same column.  This implies that 
every $x$ in $\wor(D)$ can be paired with an $x-1$ to its right.
Therefore, $\wor(T*T_\lambda)=\wor(D)$ is Yamanouchi as claimed.
\end{proof}

We define the notion of inflated weight on a set-valued tableau $S$ so 
that it reduces to Definition~\ref{def:iwt} when $S$ has no multicells.

\begin{definition} 
The \emph{inflated weight} of  a set-valued tableau $S$ 
is defined iteratively from its rows $S_1, S_2,\ldots, S_\ell$,
read from bottom to top.  With $r=2$ and $\hat S=S_1$,
starting from the last cell $c$ in $S_r$, modify $S_r*\hat S$ as follows:

\begin{enumerate}
\item move the largest entry $e$ in cell $c$ to the rightmost {\it empty} cell $(r,j)$
such that all entries in cells $(r',j)$ for all $r'<r$ are smaller than $e$.
If no such cell exists, $e$ remains in place.

\item move the next largest entry $e_2$ in cell $c$
to the rightmost cell $(r,\hat j)$, where $\hat j\leq j$,
such that all entries in cells $(r',\hat j)$ for all $r'<r$ are less than $e_2$.
Let $j=\hat j$ and repeat this step on all remaining entries in cell  $c$.

\item repeat from step 1 on the cell $\hat c$ just west of cell $c$.
\end{enumerate}

\noindent 
When all cells in row $r$ have been addressed, iterate by setting the resulting filling 
to $\hat S$ and $r :=r+1$.
The process terminates when $r=\ell$, 
with a column-strict set-valued filling $\hat S$
called the {\it inflated weight tableau} of $S$.
The inflated weight is the conjugate of the partition rearrangement of
the list of maximal entries in uncovered cells of $\hat S$. 
\end{definition}

\begin{example}
$$
S=\tiny\tableau[mcY]{48&9\cr1&5} \mapsto \tableau[mcY]{48&9\cr1&5}
\qquad
\implies\fwt(S)=(9,8)'=(2,2,2,2,2,2,2,2,1)
$$
$$
S= \tiny \tableau[mcY]{478&9\cr1&5&6} \mapsto \tableau[mcY]{478&\bl&9\cr1&5&6}\mapsto
\tableau[mcY]{4&78&9\cr1&5&6}
\implies\fwt(S)=(9,8,4)'=(3,3,3,3,2,2,2,2,1)
$$
\end{example}

\begin{remark}
The inflated weight of a set-valued tableau could also have been defined as an iterative process on its columns. Reading the columns $S_1,S_2,\ldots,S_{\lambda_1}$ from left to right, first consider just the sub-diagram with columns $S_{\lambda_1-1}$ and $S_{\lambda_1}$.  Perform  the moves described above starting with the lowest cell of $S_{\lambda_1-1}$ and working upwards. Repeat this process with the cells of $S_{\lambda_1-2}$ and proceed in this fashion until all columns are exhausted.
Note that any cell $c=(r,j)$ in a set-valued tableau has 
entries which are strictly larger than the entries in
cell $(r',j')$ for $r'\leq r, j'\leq j$. 
If we instead construct the inflated tableau by rows, 
all entries of $c$ necessarily move further to the right 
than the entries in  $(r',j')$.  With this in mind one can see the resulting inflated weights to be equal.
\end{remark}


\begin{prop}
\label{prop:svtdwtyam}
For partition $\lambda$ and set-valued tableau $S$,
if $S*T_\lambda$ is column Yamanouchi then
$\wt(S*T_\lambda)=\fwt(S*T_\lambda)$ and the columns of
the inflated weight tableau of $S*T_\lambda$ are 
interval valued.
\end{prop}
\begin{proof}
Once again the results hold for $T_\lambda$ since the rows in its inflated
weight tableau are {\it right} justified and row $i$ is made up of
$\lambda_i$ $i$'s.
We thus consider the column construction of the inflated weight tableau 
$D$ of a column Yamanouchi set-valued tableau $S*T_\lambda$.  Let $b$ be the largest entry in the most southeastern cell of $S$. Starting with the
inflated weight tableau $\hat D$ of $T_\lambda$, note that the
Yamanouchi condition on $S*T_\lambda$ implies there are strictly 
more $b-1$'s than $b$'s. At least one $b-1$ is uncovered since columns 
of $\hat D$ are interval valued.  Therefore, 
the rightmost column $j$ with an uncovered $b-1$ is the column 
into which $b$ is moved and it remains interval valued.  
If $b$ came from a multicell in $S$, an entry $a<b$ moves next.
The order in which entries are read from a multicell 
when checking the Yamanouchi property ensures that there is 
an uncovered $a-1$ in $\hat D$ (in some column $\hat j<j$).
Entry $a$ is thus placed in $(\lambda_\ell+1,\hat j)$ and
columns retain interval values.   The 
argument applies to entries taken from biggest to smallest in multicells
and then moving up the column.  Note that the columns of
$D$ are not only interval valued, but there are no multicells.

Since $D$ has no multicells, it also corresponds to the inflated
weight tableau of $T*T_\lambda$ where $T$ is the tabloid with
row $i$ made up of the entries of row $i$ in $S$.
The proof of Proposition~\ref{lem:tabdweightyam} shows
that $D$ has interval valued columns if and only if
$\fwt(T*T_\lambda)=\wt(T*T_\lambda)$.  The claim follows
by noting that $\fwt(T*T_\lambda)=\fwt(D)$ and $\wt(T*T_\lambda)=\wt(S*T_\lambda)$.
\end{proof}

Inflated weight also exposes a correspondence between reverse plane partitions and a 
distinguished subset of skew tabloid.
Precisely, for composition $\alpha$ and partitions $\lambda$ and $\nu$,
let
$$
\T_\lambda(\alpha,\fwt=\nu) = \{T*T_\lambda
: T\in\T(\alpha)\;\text{and}\;
\fwt(T*T_\lambda)=\nu\}\,.
$$
We refer to $T\in \T_\lambda$ as a {\it $\lambda$-augmented tabloid}. 
The subset $\SSYT_{\!\!\lambda}$ of elements in $\mathcal T_\lambda$ 
which are semi-standard tableaux will be of particular interest.

\begin{definition}
For any $R\in\RPP(\cdot/\lambda\,,\alpha)$, let $T$ be the tabloid of shape $\alpha$ 
where entries of row $x$ record the row positions of the subset of $x\in R$
which do not lie below an $x$ in the same column, for $x=1,\ldots,\ell(\alpha)$.
Set $\partial(R)=T*T_\lambda$.
\end{definition} 

\begin{example}
\[ \partial \left(\,\tiny{\tableau[mcY]{\textcolor{red}{1}\cr1&\textcolor{orange}{3}\cr1&\textcolor{red}{1}&\textcolor{blue}{2}\cr\bl&\bl&2&\textcolor{blue}{2}}}\,\right)= \tiny{\tableau[mcY]{\textcolor{orange}{3}\cr\textcolor{blue}{1}&\textcolor{blue}{2}\cr\textcolor{red}{2}&\textcolor{red}{4}\cr\bl&\bl&1&1}}\]
\end{example}

\  \\


\begin{prop}
\label{lem:genomic}
For any skew partition $\nu/\lambda$ and composition $\alpha$, 
the map $\partial$ gives a bijection 
$$\RPP(\nu/\lambda,\alpha)
\quad
\longleftrightarrow
\quad
\T_\lambda(\alpha,\fwt=\nu)\,.
$$
\end{prop}

\begin{proof}
Let $R\in\RPP(\nu/\lambda,\alpha)$ and $\partial(R)=T*T_\lambda$. We will first show that $\partial$ produces a tabloid with the correct shape and inflated weight. The $i\th$ component of the weight $\alpha$ of $R$ is determined by counting the number 
of columns containing an $i$.  On the other hand, one $i$ from each such column
contributes to an element in row $i$ of $T$ proving the shape of $T$ is indeed
$\alpha$.  

We next prove that $\fwt(T*T_\lambda)=\nu$ by iteration 
on $\partial(R^i)$, where $R^i$ is the sub-reverse plane partition
obtained by deleting all letters larger than $i$ from $R$.
We shall prove that $\fwt(\partial(R^i))=\nu^i$ where
$\nu^i/\lambda=\shape(R^i)$ by showing that the uncovered entries
in the inflated weight tableau $T^i$ of $\partial(R^i)$
match the column heights of $\nu^i$.

The base case is $\partial(R^0)=T_\lambda$,
whose inflated weight tableau $T^0$ contains $\lambda_i$ $i$'s {\it right} justified.  
The construction of the inflated weight tableau $T^1$ of
$\partial(R^1)$ involves rows $r_1\geq \ldots\geq r_t$, 
where $r_1$ is the height of the leftmost column $c_1$ containing 
a 1 in $R^1$.  Row positions $r_2,\ldots,r_t$ are similarly obtained by 
reading columns west to east.  $T^1$ is obtained by successively 
sliding $r_1,\ldots,r_t$ into row $\ell(\lambda)+1$ atop $T^0$ as far right 
as possible without violating the column increasing condition.  
Note that $r_1$ takes position above the entry in $T^0$ which records the height 
of column $c_1$ in $T^0$.  The same argument applied to the remaining
row positions implies that uncovered entries of $T^1$ match the column heights of $\nu^1$.
The results follows by iteration.

The bijectivity of $\partial$ is seen through the construction of its inverse.
A unique reverse plane partition $R$ of shape $\nu/\lambda$
corresponds to $T*T_\lambda$ using the following construction.
Starting with row $i=1$ of $T$, place an i in the leftmost empty cell $(e,c)$ 
of $\nu/\lambda$, where $e$ is the rightmost entry in the first row of $T$. Fill 
all empty cells below $(e,c)$ in column $c$ with i as well.  Proceed by placing 
an $i$ in the leftmost empty cell $(\hat e,\hat c)$ where $\hat e$ is the next 
to last entry $\hat e$ in row $i$ of $T$.  Repeat on all entries in row $i$
of $T$ and iterate letting $i=i+1$.
\end{proof}

\subsection{Tabloids to set-valued tableaux}

The jeu de taquin operation induces a crystal graph on the set of
tabloids with fixed inflated weight.
The highest weights are simply tabloids with strictly increasing columns.

\begin{prop}[\cite{genomic}]
\label{prop:tabloidcrystal} 
For each skew partition $\nu/\lambda$, a crystal graph whose
vertices are augmented tabloids $\T_\lambda(\cdot,\fwt=\nu)$
supports $g_{\nu/\lambda}$.  The highest weights are 
semi-standard tableaux in $\T_\lambda(\cdot,\fwt=\nu)$.
\end{prop}

A different perspective is studied by Galashin in~\cite{Gal14} using
a crystal graph on reverse plane partitions.  Proposition~\ref{prop:tabloidcrystal} 
can be proven by applying the graph isomorphism $\partial$
to the reverse plane partition crystal.
The tabloid point of view is amenable to a bijection  $\phi$ of~\cite{BM12} 
associating set-valued tableaux to strict elegant fillings and this route
leads to the $K$-theoretic Littlewood-Richardson structure constants.

The map $\phi$ on set-valued tableaux is defined by iteratively eliminating
multicells through a process called {\it dilation}.
Given a set-valued tableau $S$, let $row(S)$ be the highest row containing a 
multicell.  Let $S_{>i}$ denote the subtableau formed by taking only rows of $S$ above
row $i$.  For the rightmost multicell $c$
in ${\rm row}(S)$, define $x=x(S)$ to be the largest entry in $c$. 
The {dilation} of $S$, $di(S)$, is constructed from $S$ by 
removing $x$ from $c$ and RSK-inserting $x$ into $S_{>{\rm row}(S)}$.

\begin{Ex}
Since $row(S)=2$ and $x(S)=6$, 
$$
di\left( \tiny
\tableau[mcY]{7\cr6&7\cr34&456&8\cr1&12&23&5}
\right)=
\tiny\tableau[mcY]{7&7\cr6&6\cr34&45&8\cr1&12&23&5}
$$
\end{Ex}

It was proven~\cite[Property~4.4]{BM12} that dilation
preserves Knuth equivalence.  More precisely, for any
set-valued tableau $S$,
\begin{equation}
\label{eq:setknuth}
\fw(S)\sim \fw(di(S))\,,
\end{equation}
where the \emph{set-valued row word} $\fw(S)$ 
is the word obtained by listing the entries from rows of $S$ 
as follows (starting in the highest row): first ignore the smallest entry in each 
cell and record the remaining entries in the row from right to left and 
from largest to smallest within each cell, then record the
smallest entry of each cell from left to right.

\begin{Ex} \label{Ex:SVTword}
$$
S=\tiny\tableau[mcY]{3&456\cr12&23&3} \implies \fw(S)=653432123\,.
$$
\end{Ex}

Dilation expands a set-valued tableau by reducing the number of entries 
in a given multicell by one.  The iteration of this process produces
a semi-standard tableau from a set-valued tableau.

\begin{definition}
The map $\phi$ acts on a set valued tableau $S$ by constructing the sequence 
of set-valued tableaux
\begin{equation}
\label{inflatedseq}
S = S_0 \to di(S) = S_1 \to di(S_1) = S_2 \to \dots \to S_r\end{equation}
and defining $\phi(S)$ to be the filling of
$\sh(S_r)/\sh(S)$ where cell $\sh(S_i) / \sh(S_{i-1})$ contains
the difference between the row index of this cell and ${\rm row}(S_{i-1})$.
The sequence \eqref{inflatedseq} is defined to terminate 
at the first set-valued tableau $S_r$ with no multicell.
\end{definition}

\begin{Ex} \label{Ex:EF-SVTbij}
$$
\phi \left({\tiny\tableau[mcY]{4\cr2&23\cr1&1&1234}} \right)
={\scriptsize\tableau[scY]{2&3\cr &2\cr &&1\cr && }}
$$
is constructed by recording
the sequence of dilations 
$$
    \tiny\tableau[mcY]{4\cr2&23\cr1&1&1234} \to
    \tableau[mcY]{4\cr3\cr2&2\cr1&1&1234} \to
    \tableau[mcY]{4\cr3\cr2&2&4\cr1&1&123} \to
    \tableau[mcY]{4\cr3&4\cr2&2&3\cr1&1&12} \to
    \tableau[mcY]{4&4\cr3&3\cr2&2&2\cr1&1&1}\\
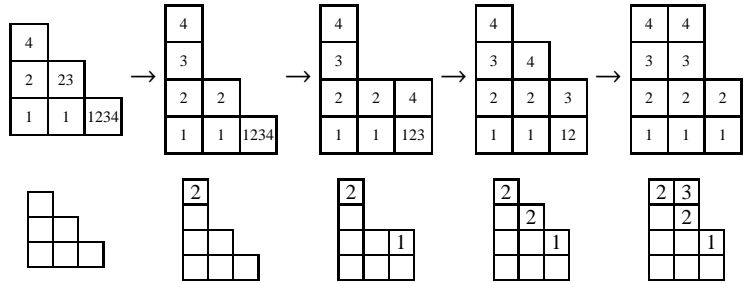
$$
$$
    \scriptsize\tableau[scY]{ \cr & \cr & & } \hspace{.5 cm}
\qquad
    \tableau[scY]{2\cr \cr & \cr & & } \hspace{.5 cm}
\qquad
    \tableau[scY]{2\cr \cr & & 1\cr & & } \hspace{.5 cm}
\qquad
    \tableau[scY]{2\cr &2\cr & & 1\cr & & } \hspace{.5 cm}
\qquad
    \tableau[scY]{2&3\cr &2\cr & &1\cr & & }
$$
\end{Ex}

The restriction of $\phi$ to an action on a subset of set-valued tableaux 
determined by their set-valued row words and a tableau $T$ is a bijection $\phi_T$ between 
these set-valued tableaux and elegant fillings. 
In particular this gives a bijective proof of Lenart's formula \cite[Theorem 3.2]{Lenart}.

\begin{prop}[\cite{BM12}, Proposition 5.6] 
\label{phi}
For any fixed tableau $T$ and partition $\eta\subset\shape(T)$,  
$$
\phi_T: 
\{S\in \SVT(\eta): \fw(S)\sim \wor(T)\}
\leftrightarrow 
 \EF(\shape(T)/\eta)
$$
is a bijection.
\end{prop}

Our purposes require a refinement of this result which
identifies strict elegant fillings with subsets of set-valued 
tableaux characterized by inflated weight. To state the result we define {\it $\lambda$-augmented set-valued tableaux} in a way analogous to that for tabloids. Precisely, let
$$\SVT_\lambda(\eta,\fwt=\alpha)=\{S*T_\lambda\,:\,S\in\SVT(\eta)\text{ and }\fwt(S*T_\lambda)=\alpha\}.$$

\begin{prop} 
\label{lem:dwt}
For each $T\in \SSYT_\lambda(\mu,\fwt=\alpha)$ and partition $\eta\subset\mu$,  
$$
\phi_T: 
\{S\in \SVT_\lambda(\eta,\fwt=\alpha): \fw(S)\sim \wor(T)\} \leftrightarrow 
 \EF(\shape(T)/\eta)
$$
is a bijection.
\end{prop}

\begin{proof}
By definition of $\phi_T$ and Proposition~\ref{phi}, it suffices to show 
that $\fwt(S)=\fwt(di(S))$  for set-valued tableau $S$.  The dilation of $S$ 
is an iterative process whereby a generic step involves the deletion of 
the largest entry $b$ from some multicell in row $r$, followed by the RSK-insertion 
of this entry starting in row $r+1$.  Row $r$ is taken to be the highest row
containing a multicell.  We proceed by showing how entries in the inflated 
weight tableau $D$  of a set-valued tableau $S$ differ from those in the inflated weight tableau
$\tilde D$ of the set-valued tableau  $\tilde S$ resulting from such a step.
From this, we argue that the uncovered entries are the same in $D$ and $\tilde D$.

Rows lower than row $r$ in $D$ and $\tilde D$ are identical since
these are unchanged by the step in dilation and the construction of inflated 
weight tableaux proceeds from bottom to top.  We claim next that rows $r$ 
of $D$ and $\tilde D$ are identical except for the omission of entry $b$ in 
$\tilde D$.  Let $a$ be the maximal entry smaller than $b$ in its multicell 
of $S$.  In row $r$ of $D$, if $b$ lies in column $j_r$, then $a$ is in 
column $j_r-t$ for some $t\geq 0$.  Note then that any cell $c$ 
in row $r$ between column $j_r-t$ and $j_r$ is empty and there 
is an entry $e>a$ below it.
Therefore, deleting $b$ from row $r$ of $D$ and reconstructing
the dilation leaves all other entries in this row unchanged.

Next consider the entries  $\{x_r,\ldots,x_n\}$ lying in the cells of
the bumping route created by the deletion of $x_r=b$ from row $r$ of $S$
followed by its insertion into row $r+1$.  In particular, $x_i$ lies in 
row $i$ of $S$ whereas it lies in row $i+1$ of of $\tilde S$.
Now consider the inflated tableau $D$ and let $c_i=(i,j_i)$ denote the 
leftmost cell of row $i$ in $D$ containing $x_i$, for $i=r,\ldots,n$.  
We prove that the only difference between row $i+1$ in $\tilde D$ and $D$
is 

(a) entry $x_{i+1}$ is replaced by $x_i$ in $c_{i+1}$ when $j_{i+1}\leq j_i$

(b) entry $x_{i+1}$ is deleted and $x_i$ is placed in cell $(i+1,j_{i})$ 
above $c_{i}$ in $\tilde D$ otherwise.

The inflated tableau $D$, with $x_{r+1}$ in column $j_{r+1}$, has an entry $e_1<x_{r+1}$ 
in column $j_{r+1}$ and entry $e_2\geq x_{r+1}$ in $(\hat r,j_{r+1}+1)$ for 
some $\hat r\leq r+1$.  By the previous discussion, since $x_r<x_{r+1}$,
$e_1$ and $e_2$ occur in $\tilde D$ in their same positions
unless $e_1=x_r$ is in $c_r$.
Note that $j_{r+1}\leq j_r$ occurs precisely when $e_1<x_r$ or $e_1=x_r$ is in $c_r$.
In this case, since $x_{r+1}$ is the leftmost entry larger than $x_r$ in row $r+1$ of $D$,
the entry $x_r$ replaces $x_{r+1}$ in position $c_{r+1}$ of $\tilde D$.  
When $j_{r+1}>j_r$, $e_1\geq x_r$ and cells of row $r+1$
and columns $j_r,j_{r}+1,\ldots,j_{r+1}-1$ must be empty
since $c_{r+1}$ contains the the leftmost entry larger than $x_r$ 
in row $r+1$ of $D$.  We see then that going from $D$ to $\tilde D$
in row $r+1$, $x_{r+1}$ is deleted and $x_r$ takes the position in column $j_r$.
The claim for generic $i$ follows by iteration.

It thus remains to show that the specified change from $D$
to $\tilde D$ does not change the set of uncovered entries.
Since every column $j\neq j_i$ is the same in $D$ and $\tilde D$, we need 
only consider columns $j_i$.
For each $t$ such that $j_t<j_{t+1}$, columns $j_t,j_t+1,\ldots,j_{t+1}-1$ 
are empty in row $t+1$ of $D$ implying that $x_t$ lies below the empty cell $(t+1,j_t)$ 
in $D$.
In $\tilde D$, $x_t$ simply slides up into this cell and the status of covered entries
is unaltered.  Note if an empty cell lies above $c_{t+1}$ in $D$,
$(b)$ applies for $i=t+1$ and results in sliding $x_{t+1}$ up into 
this empty cell in $\tilde D$.

For $t$ where $j_t\geq j_{t+1}$, $x_t$ lies just below an entry in row 
$t+1$ of $D$ whereas it lies in position $c_{t+1}=(t+1,j_{t+1})$ of $\tilde D$.
Again, our only worry then is if an empty cell lies above $c_{t+1}$ in $D$.
However, if it does, $(b)$ applies with $i=t+1$ and results in 
moving $x_{t+1}$ above $x_t$ in $\tilde D$.
\end{proof}

\section{Dual approach to the $K$-theoretic LR rule}

The preceding maps come together to identify the $g$-expansion of $g_{\nu/\lambda}$ 
without soliciting
the help of Theorem~\ref{thm:20} and Proposition~\ref{prop:21}. The last piece
of the puzzle is an appropriate sign reversing involution 
on $\lambda$-augmented set-valued tableaux. 

\begin{definition} (see also~\cite[Lemma 3]{IS14}). 
For partitions $\lambda$ and $\eta$, the map 
$$\tau:\SVT_\lambda(\eta)\to\SVT_\lambda(\eta)$$
acts on $S*T_\lambda$ as follows;
when $S$ is column $\lambda$-Yamanouchi, let $\hat{S}=S$.
Otherwise, define $c$ to be the rightmost column such that $S_{\geq c}$ 
is not column $\lambda$-Yamanouchi, where S$_{\geq c}$ is the set-valued tableau 
made up of only those cells which are weakly right of column $c$. 
Let $y$ be the rightmost letter in the column 
word $\wor(S_{\geq c})\wor(T_\lambda)$ with the property that there
are more $y$'s than $y-1$'s and take $r$ to be the row of the cell 
$(r,c)$ in $S$ containing this $y$. Denote by $cell_{\min}$ the leftmost 
cell in row $r$ containing $y$.  The image $\hat S=\tau(S)$ is defined 
by deleting $y-1$ if it is
present in $cell_{\min}$ and otherwise $\hat S$ is obtained by adding
$y-1$ to $cell_{\min}$. 
\end{definition}

\begin{example}
$$
S*T_{(2,1)}= \tiny{\tableau[mcY]{57&\textcolor{red}{7}\cr \bl&\bl& 2 \cr \bl&\bl&1&1 \cr 
\bl & \bl w(S_{\geq c})w(T_{(21)})=57\underline{\textcolor{red}{7}}211 }
}
\quad
\quad
\leftrightarrow\quad
\quad
\hat{S}*T_{(2,1)}= 
\tiny\tableau[mcY]{567&\textcolor{red}{7}\cr \bl&\bl& 2 \cr \bl&\bl&1&1\cr
\bl & \bl w(\hat{S}_{\geq c})w(T_{(21)})=567\underline{\textcolor{red}{7}}211}
$$
\end{example}

\begin{lemma}
\label{lem:srev}
For partitions $\eta, \nu,$ and $\lambda\subset \nu$,
$\tau$ is a sign-reversing involution on
$\SVT_\lambda(\eta,\fwt=\nu)$ where
set-valued tableaux with a Yamanouchi column word are fixed points. 
\end{lemma}
\begin{proof}
For $S\in \SVT_\lambda(\eta,\fwt=\nu)$, we first show that
$\hat S=\tau(S)$ is a set-valued tableau.
Certainly deleting a letter maintains the column-strict condition.  
On the other hand, if $\hat S$ is obtained by adding $y-1$,
since it is added to the leftmost cell containing a $y$ in some 
row, the rows remain non-decreasing.
Moreover, since $c$ is the rightmost column where $S_{\geq j}$
fails to be column $\lambda$-Yamanouchi (and it fails at $y$),
either $y-1$ is not in column $c$ or it lies in the cell with $y$ in
row $r$.  Therefore, the entries in cell $(r-1,c)$ is strictly smaller
than $y-1$. This implies that entries below the $cell_{\min}$  are also 
smaller than $y-1$  and we see the columns remain
strictly increasing under the action of $\tau$.

Observe further that $cell_{\min}(\hat S)=cell_{\min}(S)$ and $y(S)=y(\hat S)$ since 
any change in column word due to $\tau$ occurs to the left of $y$. 
Thus by definition of $\tau$, if $S$ is  not column $\lambda$-Yamanouchi, neither
is $\hat S$.

It thus remains to show that $\tau$ preserves inflated weight. 
By assumption, $S*T_\lambda$ first fails to be column Yamanouchi 
with an entry $y$ in cell $(r,c)$. Therefore, $S_{>c}*T_\lambda$ 
is column Yamanouchi and contains the same number of $y$'s as $y-1$'s.  
By Proposition~\ref{prop:svtdwtyam}, the inflated weight tableau of $S_{>c}*T_\lambda$
has interval valued columns implying that each $y-1$ is covered. 
 This is also true 
in the continued column construction of the inflated weight 
tableau with entries in column $c$ until reaching cell $(r,c)$.
However, at cell $(r,c)$, $y$ slides over and takes a position above
an entry strictly smaller than $y-1$ since there are no uncovered $y-1$'s.
Since there are no $y-1$'s in rows below $r$ and in any column west
of $c+1$, this is true for all the $y$'s in row $r$.
Note that the inflated weight tableaux before and after applying $\tau$
are the same up until this point.
Then, if $cell_{\min}$ contains a $y$ and $y-1$, it shares
a cell in the inflated weight tableau $S*T_\lambda$
and thus the deletion of that $y-1$ would not affect the inflated weight.
On the other hand, when $cell_{\min}$ does not contain a $y-1$ but 
the action of $\tau$ introduces a one,
we see that in the corresponding inflated weight tableau it would
slide into the cell with $y$ again preserving inflated weight.
\end{proof}

\begin{theorem}
For partitions $\mu$ and $\lambda\subset \nu$,
the coefficient of $g_\mu$ in $g_{\nu/\lambda}$ is $(-1)^{|\nu|-|\lambda|-|\mu|}$ 
times the number of column $\lambda$-Yamanouchi set-valued tableaux
of shape $\mu$ and weight $\nu/\lambda$.
\end{theorem}
\begin{proof}
Proposition~\ref{lem:genomic} allows us to convert reverse plane partitions 
of fixed shape into a class of augmented tabloids of fixed inflated weight, implying that
$$
g_{\nu/\lambda} =  
\sum_{A\in \T_\lambda(\cdot,\fwt=\nu)} \bx^{\shape(A)}
\,.
$$
The crystal graph of Proposition~\ref{prop:tabloidcrystal}  implies that
\begin{equation}
g_{\nu/\lambda} = \sum_{T\in\SSYT_\lambda(\cdot,\fwt=\nu)}
s_{\shape(T)}\,.
\end{equation}
Lenart's elegant formula \eqref{sing} then converts this identity to
an expansion in terms of the basis of $g$-functions:
$$
g_{\nu/\lambda}=
\sum_{T\in \SSYT_\lambda(\cdot,\fwt=\nu)}
\sum_{\eta\subset\shape(T)}
\sum_{E\in\EF(\shape(T)/\eta)} 
(-1)^{|\shape(T)|-|\eta|} 
\,g_{\eta}
\,.
$$
For each tableau $T$, we use the bijection $\phi_T$
and Proposition~\ref{lem:dwt} to
replace elegant fillings by set-valued tableaux.
$$
g_{\nu/\lambda}=
\sum_{T\in \SSYT_\lambda(\cdot,\fwt=\nu)}
\sum_{\eta\subset\shape(T)}
\sum_{S\in \SVT_\lambda(\eta,\fwt=\nu) \atop \fw(S)\sim \wor(T)}
(-1)^{|\shape(T)|-|\eta|} \,g_{\eta}
$$
We can replace $|\shape(T)|=|\wt(T)|$ by $|\wt(S)|$
since any tableau $T$ and set-valued tableau $S$ where
$\wor(T)\sim \fw(S)$ must have matching weights.
In fact, we can drop the dependence on $T$ altogether.
Namely, the expansion of a given set-valued tableau $S$ 
by iterated dilation gives rise to a unique tableau $T$
and $\fw(S)\sim \wor(T)$ by~\eqref{eq:setknuth}.
Moreover, our proof of Proposition~\ref{lem:dwt} 
shows that the inflated weight of $T$ 
matches $\fwt(S)$ under dilation.
Therefore,
$$
g_{\nu/\lambda}=
\sum_{\eta}\sum_{S\in\SVT_\lambda(\eta, \fwt=\nu)}
(-1)^{|\wt(S)|-|\eta|} \,g_{\eta}
\,.
$$

Finally, we apply the the sign-reversing involution $\tau$ on
$\SVT_\lambda(\eta,\fwt=\nu)$ and use Lemma~\ref{lem:srev}.
Proposition~\ref{prop:svtdwtyam} then assures us that
the fixed points $S$ satisfy $\wt(S*T_\lambda)=\nu$.
\end{proof}

%
%

 \section*{Acknowledgements}
 \label{sec:ack}
The authors thank Ryan Kaliszewski for many discussions about related work
and the reviewers for suggestions on an improved exposition.


\bibliographystyle{alpha}
\bibliography{FPSACskewg}

\begin{thebibliography}{McN06}

\bibitem[BM12]{BM12}
J.~Bandlow and J.~Morse.
\newblock Combinatorial expansions in {$K$}-theoretic bases.
\newblock {\em Electron. J. Combin.}, 19(4):Paper 39, 27, 2012.

\bibitem[Buc02]{Buch}
A.~S. Buch.
\newblock A {L}ittlewood-{R}ichardson rule for the {$K$}-theory of
  {G}rassmannians.
\newblock {\em Acta Math.}, 189(1):37--78, 2002.

\bibitem[FK94]{FK94}
S.~Fomin and A.~N. Kirillov.
\newblock Grothendieck polynomials and the {Y}ang-{B}axter equation.
\newblock In {\em Formal power series and algebraic combinatorics/{S}\'eries
  formelles et combinatoire alg\'ebrique}, pages 183--189. DIMACS, Piscataway,
  NJ, 1994.

\bibitem[Ful97]{Ful}
W.~Fulton.
\newblock {\em Young tableaux}, volume~35 of {\em London Mathematical Society
  Student Texts}.
\newblock Cambridge University Press, Cambridge, 1997.
\newblock With applications to representation theory and geometry.

\bibitem[Gal14]{Gal14}
P.~Galashin.
\newblock {A {L}ittlewood-{R}ichardson rule for dual stable {G}rothendieck
  polynomials}.
\newblock {\em ArXiv e-prints}, December 2014.

\bibitem[IS14]{IS14}
T.~Ikeda and T.~Shimazaki.
\newblock A proof of {K}-theoretic {L}ittlewood-{R}ichardson rules by
  {B}ender-{K}nuth-type involutions.
\newblock {\em Math. Res. Lett}, 21(2):333--339, 2014.

\bibitem[KM]{genomic}
R.~Kaliszewski and J.~Morse.
\newblock Cycloids for equivariant ${K}$-theory of {G}rassmannians.

\bibitem[Knu70]{K70}
D.~E. Knuth.
\newblock Permutations, matrices, and generalized {Y}oung tableaux.
\newblock {\em Pacific J. Math.}, 34:709--727, 1970.

\bibitem[Len00]{Lenart}
C.~Lenart.
\newblock Combinatorial aspects of the {$K$}-theory of {G}rassmannians.
\newblock {\em Ann. Comb.}, 4(1):67--82, 2000.

\bibitem[LP07]{LP07}
T.~Lam and P.~Pylyavskyy.
\newblock Combinatorial {H}opf algebras and {$K$}-homology of {G}rassmannians.
\newblock {\em Int. Math. Res. Not. IMRN}, (24):Art. ID rnm125, 48, 2007.

\bibitem[LR34]{LR}
D.~E. Littlewood and A.~R. Richardson.
\newblock Group characters and algebra.
\newblock {\em Philosophical Transactions of the Royal Society of London A:
  Mathematical, Physical and Engineering Sciences}, 233(721-730):99--141, 1934.

\bibitem[LS81]{LSmonoid}
A.~Lascoux and M.-P. Sch{\"u}tzenberger.
\newblock Le mono\"\i de plaxique.
\newblock In {\em Noncommutative structures in algebra and geometric
  combinatorics ({N}aples, 1978)}, volume 109 of {\em Quad. ``Ricerca Sci.''},
  pages 129--156. CNR, Rome, 1981.

\bibitem[LS83]{LS}
A.~Lascoux and M.-P. Sch{\"u}tzenberger.
\newblock Symmetry and flag manifolds.
\newblock In {\em Invariant theory ({M}ontecatini, 1982)}, volume 996 of {\em
  Lecture Notes in Math.}, pages 118--144. Springer, Berlin, 1983.

\bibitem[McN06]{MC06}
P.-J. McNamara.
\newblock Factorial {G}rothendieck polynomials.
\newblock {\em Journal of Combinatorics}, 13(3):R71, 2006.

\bibitem[PY15]{PY15}
O.~{Pechenik} and A.~{Yong}.
\newblock {Equivariant {$K$}-theory of {G}rassmannians}.
\newblock {\em ArXiv e-prints}, June 2015.

\bibitem[Rob38]{R38}
G.~de~B. Robinson.
\newblock On the {R}epresentations of the {S}ymmetric {G}roup.
\newblock {\em Amer. J. Math.}, 60(3):745--760, 1938.

\bibitem[RS98]{RS98}
J.~B. Remmel and M.~Shimozono.
\newblock A simple proof of the {L}ittlewood-{R}ichardson rule and
  applications.
\newblock {\em Discrete Math.}, 193(1-3):257--266, 1998.
\newblock Selected papers in honor of Adriano Garsia (Taormina, 1994).

\bibitem[Sch61]{S61}
C.~Schensted.
\newblock Longest increasing and decreasing subsequences.
\newblock {\em Canad. J. Math.}, 13:179--191, 1961.

\bibitem[Sch77]{S77}
M.-P. Sch{\"u}tzenberger.
\newblock La correspondance de {R}obinson.
\newblock In {\em Combinatoire et repr\'esentation du groupe sym\'etrique
  ({A}ctes {T}able {R}onde {CNRS}, {U}niv. {L}ouis-{P}asteur {S}trasbourg,
  {S}trasbourg, 1976)}, pages 59--113. Lecture Notes in Math., Vol. 579.
  Springer, Berlin, 1977.

\bibitem[Sta99]{EC2}
R.~P. Stanley.
\newblock {\em Enumerative combinatorics. {V}ol. 2}, volume~62 of {\em
  Cambridge Studies in Advanced Mathematics}.
\newblock Cambridge University Press, Cambridge, 1999.
\newblock With a foreword by Gian-Carlo Rota and appendix 1 by Sergey Fomin.

\bibitem[Ste02]{Ste02}
J.~R. Stembridge.
\newblock A concise proof of the {L}ittlewood-{R}ichardson rule.
\newblock {\em Electron. J. Combin.}, 9(1):Note 5, 4 pp. (electronic), 2002.

\end{thebibliography}

\end{document}